\documentclass{article}%
\usepackage{graphicx}
\usepackage{amsmath}
\usepackage{amsfonts}
\usepackage{amssymb}
\usepackage{setspace}
\usepackage[hmargin=3.5cm,vmargin=3.5cm]{geometry}%
\providecommand{\U}[1]{\protect\rule{.1in}{.1in}}

\newtheorem{theorem}{Theorem}
{}
\newtheorem{acknowledgement}{Acknowledgement}

\newtheorem{corollary}{Corollary}

\newtheorem{definition}{Definition}

{}
\newtheorem{notation}{Notation}

\newtheorem{proposition}{Proposition}
\newtheorem{remark}{Remark}

\newenvironment{proof}[1][Proof]{\textbf{#1.} }{\ \rule{0.5em}{0.5em}}

\begin{document}

\title{On the Band Functions and Bloch Functions.}
\author{O. A. Veliev\\{\small Dogus University, \ Istanbul, Turkey.}\ \\{\small e-mail: oveliev@dogus.edu.tr}}
\date{}
\maketitle

\begin{abstract}
In this paper, we consider the continuity of the band functions and Bloch
functions of the differential operators generated by the differential
expressions with periodic matrix coefficients.



\end{abstract}

\section{Introduction}

In this paper, we consider the continuity of the Bloch eigenvalues, band
functions and Bloch functions with respect to the quasimomentum of the
differential operator $T$, generated in the space $L_{2}^{m}(\mathbb{R}^{d})$
by formally self-adjoint differential expression
\begin{equation}
Tu=%
{\textstyle\sum\limits_{\left\vert \alpha\right\vert =2s}}
Q_{\alpha}D_{\alpha}u+%
{\textstyle\sum\limits_{\left\vert \alpha\right\vert \leq2s-1}}
Q_{\alpha}(x)D_{\alpha}u+Bu, \tag{1}%
\end{equation}
where $d\geq1$, $\alpha=(\alpha_{1},\alpha_{2},...,\alpha_{d})$ is a
multi-indeks, $\left\vert \alpha\right\vert =\alpha_{1}+\alpha_{2}+\cdot
\cdot\cdot+\alpha_{d}$,
\[
D_{\alpha}=\left(  \frac{1}{i}\frac{\partial}{\partial x_{1}}\right)
^{\alpha_{1}}\left(  \frac{1}{i}\frac{\partial}{\partial x_{2}}\right)
^{\alpha_{2}}...\left(  \frac{1}{i}\frac{\partial}{\partial x_{d}}\right)
^{\alpha_{d}},
\]
$Q_{\alpha}$, for each $\alpha$, is an $m\times m$ matrix. Here the entries of
$Q_{\alpha}(x)$, for $\left\vert \alpha\right\vert \leq2s-1$, are bounded
functions that are periodic with respect to a lattice $\Omega$, the entries of
$Q_{\alpha}$, for $\left\vert \alpha\right\vert =2s$, are real numbers and $B$
is a bounded operator commuting with the shift operators $S_{\omega
}:u(x)\rightarrow u(x+\omega)$, for $\omega\in\Omega$. Note that $L_{2}%
^{m}(\mathbb{R}^{d})$ is the space of the vector-valued functions $f=\left(
f_{1},f_{2},...,f_{m}\right)  $ with $f_{k}\in L_{2}(\mathbb{R}^{d})$, for
$k=1,2,...,m$.

To describe briefly the scheme of this paper, let us introduce the following
notations. Let $\Gamma$ be the lattice dual to $\Omega$. Denote by $F$ and
$F^{\star}$ the fundamental domains of the lattices $\Omega$ and $\Gamma$,
respectively. Let$\ T_{t}$ be the operator generated in $L_{2}^{m}(F)$ by (1)
and the quasiperiodic conditions%
\begin{equation}
u(x+\omega)=e^{i\left\langle t,\omega\right\rangle }u(x),\ \forall\omega
\in\Omega, \tag{2}%
\end{equation}
where $t\in F^{\star}$ and $\left\langle \cdot,\cdot\right\rangle $ is the
inner product in $\mathbb{R}^{d}$. Denote by $SBC(H)$ the set of below-bounded
self-adjoint operators, with compact resolvents acting in the Hilbert space
$H$. If $\ $%
\begin{equation}
T_{t}\in SBC(L_{2}^{m}(F)), \tag{3}%
\end{equation}
then the spectrum of $T_{t}$ consists of the eigenvalues. Let $\lambda
_{1}(t),\lambda_{2}(t),...$ be the eigenvalues of $T_{t}$ numerated in the
nondecreasing order
\begin{equation}
\lambda_{1}(t)\leq\lambda_{2}(t)\leq\cdot\cdot\cdot. \tag{4}%
\end{equation}
The eigenvalues and eigenfunctions of $T_{t}$, for $t\in F^{\star}$, are
called the Bloch eigenvalues and Bloch functions of $T$, respectively. The
function $\lambda_{n}:F^{\star}\rightarrow\mathbb{R}$ is said to be the $n$-th
band function of $T$.

To investigate the continuity of the band functions and Bloch functions of
$T$, first we study the continuity of the eigenvalues and eigenfunctions for
the family $\left\{  A_{t}:t\in E\right\}  $ of the operators $A_{t}\in
SBC(H)$, where $E$ is a metric space and $A_{t}$ continuously depends on $t\in
E$ in the generalized sense defined in [2, Chap. 4, page 202]. For the
simplicity of the notations, the parameter of the family $\left\{  A_{t}:t\in
E\right\}  $ is also denoted by $t$, the eigenvalues of $A_{t}$ also are
denoted by $\lambda_{1}(t),\lambda_{2}(t),...$ and are numerated in the
nondecreasing order as in the case $T_{t}$ (see (4)). In Theorems 1-3, we
prove that the eigenvalue $\lambda_{n}(t)$ and corresponding eigenspace of
$A_{t}$ continuously depend on $t\in E$. Moreover, we prove that, if
$\lambda_{n}(t_{0})$ is a simple eigenvalue, then the eigenvalues $\lambda
_{n}(t)$ are simple in some neighborhood of $t_{0}$ and the corresponding
normalized eigenfunctions $\Psi_{n,t}$ of $A_{t}$ can be chosen so that
$\left\Vert \Psi_{n,t}-\Psi_{n,t_{0}}\right\Vert \rightarrow0$ as
$t\rightarrow t_{0}$, where $\left\Vert \cdot\right\Vert $ is the norm of $H$.
Then, we use these results for the family $\left\{  T_{t}:t\in F^{\star
}\right\}  $ of the operators $T_{t}.$ Namely, in Theorem 4 we prove that, if
(3) and an additional condition (10) hold, then $T_{t}$ continuously depends
on $t\in F^{\star}$ in the generalized sense and hence the results obtained
for $A_{t}$ continue to hold for $T_{t}$. In order not to deviate from the
purpose of this paper, we do not discuss the conditions on (1) under which (3)
and (10) hold, since the consideration of these conditions is long technical.
Nevertheless, in the end of this paper (see Corollary 1), we give an example
which shows that the obtained results include the continuity of the band
functions and Bloch functions of a large class of the differential operators
generated by (1).

Note that, as was noted in the physical and mathematical literatures (see for
example the books [1, 3, 5] and paper [4]), the continuity of the band
functions of the Schr\"{o}dinger operator with a periodic potential and some
other periodic operators is well-known or immediately follows from the
perturbation theory described in [2]. Here we do not discuss all the numerous
investigations about this, and give only the detailed proof of the continuity
of the band functions and Bloch functions for the large class of the systems
of differential operators by using the generalized convergence in the sense of
[2]. Finally, we note that, using the results of [Chapter 13 of [3]] in a
standard way, one can easily verify that $T$\ can be expressed as the direct
integral of the operators $T_{t}$, for $t\in F^{\star}$ and that the spectrum
of the operator $T$ is the union of the spectra of the operators $T_{t}$ for
$t\in F^{\star}$.

\section{Main Results}

First, we consider the general operators by using the generalized convergence
of the closed operators defined in [2, Chap. 4, page 202]. Let us state the
definition of the generalized convergence in the notation which is suitable
for this investigation. Let $A$ and $B$ be self-adjoint operators in the
Hilbert space $H$. Since\ the self-adjoint operators are closed, the graphs
$G(A)=\left\{  \left(  u,Au\right)  :u\in D(A)\right\}  $ and $G(B)=\left\{
\left(  u,Bu\right)  :u\in D(B)\right\}  $ of these operators are closed
linear manifolds of the Hilbert space $H^{2}$. Let $S_{A}$ and $S_{B}$ be
respectively the unit spheres of $G(A)$ and $G(B)$. The gap $g(T,S)$ between
the self-adjoint operators $A$ and $B$ is defined as follows:
\begin{equation}
g(A,B)=\max\left\{  \sup_{v\in S_{A}}\left(  \inf_{w\in G(B)}\left\Vert
w-v\right\Vert \right)  ,\sup_{v\in S_{B}}\left(  \inf_{w\in G(A)}\left\Vert
w-v\right\Vert \right)  \right\}  . \tag{5}%
\end{equation}
We use the following definitions and theorems of [2].

\begin{definition}
(See [2, Chap. 4, page 202]). We say that, the sequence of the closed
operators $A_{n}$ converges to the closed operator $A$ in the generalized
sense, if $g(A_{n},A)\rightarrow0$ as $n\rightarrow\infty$.
\end{definition}

This convergence determines the following definition of the continuity.

\begin{definition}
We say that, a family of the closed operators $A_{t}$ is continuous at $E$ in
the generalized sense, if for each $t_{0}\in E$ and for any sequence $\left\{
t_{n}\right\}  \subset E$ converging to $t_{0}$, $g(A_{t_{n}},A_{t_{0}%
})\rightarrow0$ as $n\rightarrow\infty$.
\end{definition}

Besides these definitions, we use Theorems 3.1 and 3.16 of [2] (see Chap. 4,
pages 208 and 212). Let us formulate these theorems in the suitable form by
using the above notations. Theorem 3.1 of [2] states that if the closed curve
$\gamma$ belongs to the resolvent set $\rho(A)$ of $A\in SBC(H),$ then there
exists $\delta>0$ such that $\gamma\in\rho(B)$, for any operator $B\in SBC(H)$
with $g(A,B)<\delta$. Theorem 3.16 of [2] states that, in this case the
numbers of the eigenvalues (counting the multiplicity) of $A$ and $B$ lying
inside $\gamma$ are the same. Moreover,
\begin{equation}
\text{ }\left\Vert \int_{\gamma}(A-\lambda I)^{-1}d\lambda-\int_{\gamma
}(B-\lambda I)^{-1}d\lambda\right\Vert \rightarrow0\text{,} \tag{6}%
\end{equation}
as $g(A,B)\rightarrow0$.

Now using these statements, we consider a continuous family $\left\{  A_{t}\in
SBC(H):t\in E\right\}  $ of the operators by using the following notations.

\begin{notation}
In (4), the eigenvalues of the operator $A_{t_{0}}$ are denoted by counting
the multiplicity. Let us denote by $\mu_{1}(t_{0}),\mu_{2}(t_{0}),...$ the
eigenvalues of $A_{t_{0}}$ without counting the multiplicity. In other words,
$\mu_{1}(t_{0})<\mu_{2}(t_{0})<\cdot\cdot\cdot$ are the eigenvalues of
$A_{t_{0}}$ with the multiplicities $k_{1},k_{2},...$, respectively. Since
$\lambda_{n}(t_{0})\rightarrow\infty$ as $n\rightarrow\infty$, for each $n$
there exists $p$ such that $n\leq k_{1}+k_{2}+\cdot\cdot\cdot+k_{p}$. This
notation, together with the notation (4) implies that
\[
\lambda_{1}(t_{0})=\lambda_{2}(t_{0})=\cdot\cdot\cdot=\lambda_{s_{1}}%
(t_{0})=\mu_{1}(t_{0}),
\]%
\[
\lambda_{s_{1}+1}(t_{0})=\lambda_{s_{1}+2}(t_{0})=\cdot\cdot\cdot
=\lambda_{s_{2}}(t_{0})=\mu_{2}(t_{0}),
\]
and
\[
\lambda_{s_{p-1}+1}(t_{0})=\lambda_{s_{p-1}+2}(t_{0})=\cdot\cdot\cdot
=\lambda_{s_{p}}(t_{0})=\mu_{p}(t_{0}),
\]
where $s_{p}=k_{1}+k_{2}+\cdot\cdot\cdot+k_{p}$ and $n\leq s_{p}$.
\end{notation}

Now we prove the following theorems for the continuous family $\left\{
A_{t}\in SBC(H):t\in E\right\}  $.

\begin{theorem}
Let $\left\{  A_{t}:t\in E\right\}  $ be the continuous family of the
operators $A_{t}\in SBC(H)$, $t_{0}\in E$ and $\left\{  t_{k}\in
E:k\in\mathbb{N}\right\}  $ be a sequence converging to $t_{0}.$ Then, for
every $r>0$ satisfying the inequality
\begin{equation}
r<\frac{1}{2}\min_{j=1,2,...p}\left(  \mu_{j+1}(t_{0})-\mu_{j}(t_{0})\right)
, \tag{7}%
\end{equation}
there exists $N>0$ such that, each of the operators $A_{t_{k}}$, for $k>N$,
has\ $k_{j}$ eigenvalues in the interval $\left(  \mu_{j}(t_{0})-r,\mu
_{j}(t_{0})+r\right)  $, where $j=1,2,...,p$, the numbers $p$ and $\mu
_{j}(t_{0})$ are defined in Notation 1. Moreover, the eigenvalues lying in
$\left(  \mu_{j}(t_{0})-r,\mu_{j}(t_{0})+r\right)  $ are $\lambda_{s_{j-1}%
+1}(t_{k}),\lambda_{s_{j-1}+2}(t_{k}),...,\lambda_{s_{j}}(t_{k})$, where
$s_{0}=0$ and $s_{j}$'s, for $j\geq1$, are defined in Notation 1.
\end{theorem}

\begin{proof}
By (7), Definition 2 and Theorem 3.1 of [2], there exists $N>0$ such that the
circle $D_{j}(r)=\left\{  z\in\mathbb{C}:\left\vert z-\mu_{j}(t_{0}%
)\right\vert =r\right\}  $ belongs to the resolvent set of the operators
$A_{t_{k}}$, for $k>N$. Therefore,\ by Theorem 3.16 of [2], $A_{t_{k}}$ has
$k_{j}$ eigenvalues inside $D_{j}$. In the same way, we prove that $A_{t_{k}}%
$, for $k>N$, has no eigenvalues in the intervals $(-\infty,\mu_{1}(t_{0})-r]$
and $\left[  \mu_{j}(t_{0})+r,\mu_{j+1}(t_{0})-r\right]  $, for $j=1,2,...,p$,
since $A_{t_{0}}$ has no eigenvalues in those intervals. Therefore, the
eigenvalues of $A_{t_{k}}$, for $k>N$, lying in $\left(  \mu_{j}(t_{0}%
)-r,\mu_{j}(t_{0})+r\right)  $ are $\lambda_{s_{j-1}+1}(t),\lambda_{s_{j-1}%
+2}(t),...,\lambda_{s_{j}}(t)$, for $j=1,2,...,p$.
\end{proof}

Now, we are ready to prove the main results for this continuous family.

\begin{theorem}
Let $\left\{  A_{t}:t\in E\right\}  $ be the continuous family of the
operators $A_{t}\in SBC(H)$. Then, the eigenvalues (4) of $A_{t}$ continuously
depend on $t\in E$.
\end{theorem}

\begin{proof}
Consider the sequence $r_{s}\rightarrow0$. By Theorem 1, there exists $N_{s}$
such that, if $k>N_{s}$ then the eigenvalues $\lambda_{s_{j-1}+1}%
(t_{k}),\lambda_{s_{j-1}+2}(t_{k}),...,\lambda_{s_{j}}(t_{k})$ lie in $\left(
\mu_{j}(t_{0})-r_{s},\mu_{j}(t_{0})+r_{s}\right)  $. Therefore, we have
$\left\vert \lambda_{n}(t_{k})-\lambda_{n}(t_{0})\right\vert <r_{s}$, for
$k>N_{s}$, since $n\in\lbrack s_{j-1}+1,s_{j}]$ and $\lambda_{n}(t_{0}%
)=\mu_{j}(t_{0})$ due to Notation 1. Now letting $s$ tend to infinity, we
obtain that $\lambda_{n}(t_{k})\rightarrow\lambda_{n}(t_{0})$ as
$k\rightarrow\infty$, for any sequence $\left\{  t_{k}\in E:k\in
\mathbb{N}\right\}  $ converging to $t_{0}$. Thus $\lambda_{n}$ is continuous
at $t_{0}$.
\end{proof}

Now, suppose that $\lambda_{n}(t_{0})$ is a simple eigenvalue. Then, by
Theorem 1, for every $r>0$ satisfying (7), there exists $N>0$ such that the
operator $A_{t_{k}}$, for $k>N$, has a unique eigenvalue (counting the
multiplicity) in the interval $\left(  \lambda_{n}(t_{0})-r,\lambda_{n}%
(t_{0})+r\right)  $ and this eigenvalue is $\lambda_{n}(t_{k})$. Since
$\lambda_{n}(t_{k})$ is a simple eigenvalue, we have%
\begin{equation}
\int_{\gamma}(A_{t_{k}}-\lambda I)^{-1}fd\lambda=(f,\Psi_{n,t_{k}}%
)\Psi_{n,t_{k}}, \tag{8}%
\end{equation}
for $f\in H$, $k=0$, and $k>N$, where $\gamma$ is a closed curve that encloses
only the eigenvalue $\lambda_{n}(t_{k})$ and $\Psi_{n,t_{k}}$ is a normalized
eigenfunction corresponding to $\lambda_{n}(t_{k})$. Using (6) and (8), we
obtain the following relations:
\[
(f,\Psi_{n,t_{k}})\Psi_{n,t_{k}}\rightarrow(f,\Psi_{n,t_{0}})\Psi_{n,t_{0}}%
\]
as $t_{k}\rightarrow t_{0}$, for each $f\in H$. Here replacing $f$ by
$\Psi_{n,t_{0}}$, we obtain
\begin{equation}
(\Psi_{n,t_{0}},\Psi_{n,t_{k}})\Psi_{n,t_{k}}\rightarrow\Psi_{n,t_{0}} \tag{9}%
\end{equation}
and $\left\vert (\Psi_{n,t_{0}},\Psi_{n,t_{k}})\right\vert \rightarrow1$ as
$t_{k}\rightarrow t_{0}$. Since a normalized eigenfunction is still normalized
if it is multiplied by a factor of absolute value $1$, $\Psi_{k,t_{k}}$ and
$\Psi_{k,t_{0}}$ can be chosen so that
\begin{equation}
\arg(\Psi_{n,t_{0}},\Psi_{n,t_{k}})=0, \tag{10}%
\end{equation}
that is, $\left\vert (\Psi_{n,t_{0}},\Psi_{n,t_{k}})\right\vert =(\Psi
_{n,t_{0}},\Psi_{n,t_{k}})$. Thus
\[
(\Psi_{n,t_{0}},\Psi_{n,t_{k}})\rightarrow1,
\]
as $t_{k}\rightarrow t_{0}$. This, together with (9) implies that
\[
\left\Vert \Psi_{n,t_{k}}-\Psi_{n,t_{0}}\right\Vert \leq\left\Vert
(1-(\Psi_{n,t_{0}},\Psi_{n,t_{k}}))\Psi_{n,t_{k}}\right\Vert +\left\Vert
(\Psi_{n,t_{0}},\Psi_{n,t_{k}})\Psi_{n,t_{k}}-\Psi_{n,t_{0}}\right\Vert
\rightarrow0,
\]
as $t\rightarrow t_{0}$. It means that, $\Psi_{n,t}$ is continuous at $t_{0}$.
Thus, we have

\begin{theorem}
Let $\left\{  A_{t}:t\in E\right\}  $ be the continuous family of the
operators $A_{t}\in SBC(H)$ and $t_{0}\in E$. If $\lambda_{n}(t_{0})$ is a
simple eigenvalue, then the eigenvalues $\lambda_{n}(t)$ are simple in some
neighborhood of $t_{0}$ and the corresponding normalized eigenfunctions
$\Psi_{n,t}$ can be chosen so that $\left\Vert \Psi_{n,t}-\Psi_{n,t_{0}%
}\right\Vert \rightarrow0$ as $t\rightarrow t_{0}$.
\end{theorem}

Now, we consider the application of Theorems 1-3 to the differential operators
$T_{t}$ for $t\in F^{\star}$ defined in $L_{2}^{m}(F)$ by (1) and (2), where
$F$ and $F^{\star}$ is the fundamental domains of the lattices $\Omega$ and
$\Gamma.$ Without loss of generality, we assume that the measure of the
fundamental domain $F$ of the lattice $\Omega$ is $1$.

\begin{theorem}
Suppose that (3) holds and there exist $\varepsilon>0$ and $c>0$ such that
\begin{equation}
\left\Vert Tu+cu-Bu\right\Vert +\left\Vert u\right\Vert \geq\varepsilon%
{\textstyle\sum\limits_{\left\vert \alpha\right\vert \leq2s-1}}
\left\Vert D_{\alpha}u\right\Vert , \tag{11}%
\end{equation}
for all $u\in D(T_{t})$ and $t\in F^{\ast}$, where $D(T_{t})$ is the domain of
definition of $T_{t}$, $Tu$ and $B$ are defined in (1). Then $\left\{
T_{t}:t\in F^{\ast}\right\}  $ is a continuous family of the operators in the
sense of Definition 2 and hence Theorems 1-3 continue to hold for this family.
\end{theorem}

\begin{proof}
For the simplicity of the notation, denote $T_{t}+cI-B$ by $A_{t}$. By Theorem
2.23 (c) of [2, Chap. 4, page 206], it is enough to prove that, the sequence
$\left\{  A_{t_{n}}\right\}  $ converges to $A_{t_{0}}$ in the generalized
sense, for any sequence $\left\{  t_{n}\right\}  \subset F^{\ast}$ converging
to $t_{0}$. If $u\in D(A_{t_{0}})$, then $e^{i\left\langle t_{n}%
-t_{0},x\right\rangle }u\in D(A_{t_{n}})$. By the product rule of
differentiation, we have
\begin{equation}
A_{t_{n}}e^{i\left\langle t_{n}-t_{0},x\right\rangle }u=e^{i\left\langle
t_{n}-t_{0},x\right\rangle }A_{t_{0}}u+\left\vert t_{n}-t_{0}\right\vert Su,
\tag{12}%
\end{equation}
where $e^{i\left\langle t_{n}-t_{0},x\right\rangle }A_{t_{0}}u$ is the sum of
the terms of $A_{t_{n}}e^{i\left\langle t_{n}-t_{0},x\right\rangle }u$ for
which no differentiation is applied to $e^{i\left\langle t_{n}-t_{0}%
,x\right\rangle }$ and $\left\vert t_{n}-t_{0}\right\vert Su$ is the sum of
the other terms of $A_{t_{n}}e^{i\left\langle t_{n}-t_{0},x\right\rangle }%
u,$\ that is, is the sum of the terms of $A_{t_{n}}e^{i\left\langle
t_{n}-t_{0},x\right\rangle }u$ for which some differentiations are applied to
$e^{i\left\langle t_{n}-t_{0},x\right\rangle }$. It is clear that, $Su$ is the
differential expression of order $\leq2s-1$, with bounded coefficients.
Therefore, there exists $M$ such that
\begin{equation}
\left\Vert Su\right\Vert <M%
{\textstyle\sum\limits_{\left\vert \alpha\right\vert \leq2s-1}}
\left\Vert D_{\alpha}u\right\Vert . \tag{13}%
\end{equation}
Now, using (11)-(13), (5) and Definition 1, we prove that $\left\{  A_{t_{n}%
}\right\}  $ converges to $A_{t_{0}}$ in the generalized sense. Let $S_{t}$ be
the unit sphere of $G\left(  A_{t}\right)  $. If $v=(u,A_{t_{0}}u)\in
S_{t_{0}}$, then $w=(e^{i\left\langle t_{n}-t_{0},x\right\rangle }u,A_{t_{n}%
}e^{i\left\langle t_{n}-t_{0},x\right\rangle }u)\in G(A_{t_{n}})$ and
\begin{equation}
\left\Vert w-v\right\Vert \leq\left\Vert e^{i\left\langle t_{n}-t_{0}%
,x\right\rangle }u-u\right\Vert +\left\Vert A_{t_{n}}e^{i\left\langle
t_{n}-t_{0},x\right\rangle }u-A_{t_{0}}u\right\Vert . \tag{14}%
\end{equation}
For the second term of (14), we have
\begin{equation}
\left\Vert A_{t_{n}}e^{i\left\langle t_{n}-t_{0},x\right\rangle }u-A_{t_{0}%
}u\right\Vert \leq\left\Vert A_{t_{n}}e^{i\left\langle t_{n}-t_{0}%
,x\right\rangle }u-e^{i\left\langle t_{n}-t_{0},x\right\rangle }A_{t_{0}%
}u\right\Vert +\left\Vert \left(  e^{i\left\langle t_{n}-t_{0},x\right\rangle
}-1\right)  A_{t_{0}}u\right\Vert . \tag{15}%
\end{equation}
On the other hand, using (12), (13) and (11), we obtain
\begin{equation}
\left\Vert A_{t_{n}}e^{i\left\langle t_{n}-t_{0},x\right\rangle }%
u-e^{i\left\langle t_{n}-t_{0},x\right\rangle }A_{t_{0}}u\right\Vert
=\left\vert t_{n}-t_{0}\right\vert \left\Vert Su\right\Vert <\left\vert
t_{n}-t_{0}\right\vert M\left(
{\textstyle\sum\limits_{\left\vert \alpha\right\vert \leq2s-1}}
\left\Vert D_{\alpha}u\right\Vert \right)  \leq\tag{16}%
\end{equation}%
\[
\frac{\left\vert t_{n}-t_{0}\right\vert M}{\varepsilon}\left(  \left\Vert
A_{t_{0}}u\right\Vert +\left\Vert u\right\Vert \right)  .
\]
Moreover, the inclusion $(u,A_{t_{0}}u)\in S_{t_{0}}$ implies that $\left\Vert
A_{t_{0}}u\right\Vert +\left\Vert u\right\Vert <2$. This inequality, together
with (14)-(16) implies that there exists $C>0$ such that
\[
\sup_{v\in S_{t_{0}}}\left(  \inf_{w\in G(A_{t_{n}})}\left\Vert w-v\right\Vert
\right)  \leq C\left\vert t_{n}-t_{0}\right\vert .
\]
In the same way, we prove that
\[
\sup_{v\in S_{t_{n}}}\left(  \inf_{w\in G(A_{t_{0}})}\left\Vert w-v\right\Vert
\right)  \leq C\left\vert t_{n}-t_{0}\right\vert .
\]
Therefore by (5), $g(A_{t_{n}},A_{t_{o}})\rightarrow0$ as $n\rightarrow\infty
$. Hence, by Definition 2, $\left\{  T_{t}:t\in F^{\ast}\right\}  $ is a
continuous family in the generalized sense.
\end{proof}

By the standard estimation, one can find a large class of differential
operators for which (11) and (3) hold and hence Theorems 1-3 continue to hold.
For example, we find a class of the partial differential operators satisfying
(11) and (3) in the following manner. Let us write (1) in the form
\begin{equation}
Tu=Lu+Pu+Bu, \tag{17}%
\end{equation}
where
\[
Lu=%
{\textstyle\sum\limits_{\left\vert \alpha\right\vert =2s}}
Q_{\alpha}D_{\alpha}u,\text{ }Pu=%
{\textstyle\sum\limits_{\left\vert \alpha\right\vert \leq2s-1}}
Q_{\alpha}(x)u.
\]
First we note that, there exists $c_{1}>0$ for which
\begin{equation}
\left\Vert Pu\right\Vert \leq c_{1}%
{\textstyle\sum\limits_{\left\vert \alpha\right\vert \leq2s-1}}
\left\Vert D_{\alpha}u\right\Vert , \tag{18}%
\end{equation}
since the coefficients of the expression $Pu$ are matrices with bounded
entries. Here and in subsequent relations we denote by $c_{1},c_{2},...$ the
positive constants. Then, we find a condition on the main term $Lu$\ of (16)
for which, there exists $c>0$ such that
\begin{equation}
\left\Vert Lu+cu\right\Vert \geq\left(  c_{1}+\varepsilon\right)
{\textstyle\sum\limits_{\left\vert \alpha\right\vert \leq2s-1}}
\left\Vert D_{\alpha}u\right\Vert . \tag{19}%
\end{equation}
It is clear that (11) follows from (19) and (18). The condition on $L$, namely
on the principal symbol of (1) is the following. Suppose that
\begin{equation}
L=%
{\textstyle\sum\limits_{\left\vert \alpha\right\vert =2s}}
q_{\alpha}I_{m}D_{\alpha}, \tag{20}%
\end{equation}
where $I_{m}$ is $m\times m$ unit matrix, and there exists $c_{2}>0$ such
that
\begin{equation}%
{\textstyle\sum\limits_{\left\vert \alpha\right\vert =2s}}
q_{\alpha}\left(  \xi\right)  ^{\alpha}\geq c_{2}\left\vert \xi\right\vert
^{2s}, \tag{21}%
\end{equation}
for each $\xi\in\mathbb{R}^{d}$, where $q_{\alpha}\in\mathbb{R}$ and $\left(
\xi\right)  ^{\alpha}=\xi_{1}^{\alpha_{1}}\xi_{2}^{\alpha_{2}}...\xi
_{d}^{\alpha_{d}}$. Now we prove that if (20) and (21) hold, then there exists
$c>0$ such that (19) and hence (11) hold. Moreover, we prove that (3) also
holds and obtain the following consequence of Theorem 4.

\begin{corollary}
Suppose that $L$ has the form (20), inequality (21) holds, and that $Pu$ is a
formally self-adjoint differential expression. Then, the Bloch eigenvalues (4)
of the differential operator $T_{t}$ generated in $L_{2}^{m}(F)$ by (1) and
(2) continuously depend on $t\in F^{\ast}$. Moreover, if $\lambda_{n}(t_{0}%
)$\ is a simple eigenvalue, then the eigenvalues $\lambda_{n}(t)$ are simple
in some neighborhood of $t_{0}$ and the corresponding normalized
eigenfunctions $\Psi_{n,t}$ can be chosen so that $\left\Vert \Psi_{n,t}%
-\Psi_{n,t_{0}}\right\Vert \rightarrow0$ as $t\rightarrow t_{0}$.
\end{corollary}

\begin{proof}
First we prove the validity of (19) if (20) and (21) hold. Instead of (19), we
prove
\begin{equation}
\left\Vert Lu+cu\right\Vert ^{2}\geq\left(  c_{1}+\varepsilon\right)
^{2}\left(
{\textstyle\sum\limits_{\left\vert \alpha\right\vert \leq2s-1}}
\left\Vert D_{\alpha}u\right\Vert \right)  ^{2}, \tag{22}%
\end{equation}
for all $u\in L_{2}^{m}(F)$ such that $\partial^{\alpha}u\in L_{2}^{m}(F)$,
for $\left\vert \alpha\right\vert \leq2s$. It follows from (20) and (21) that
\[
\left\Vert L\left(  e^{i\left\langle \gamma,x\right\rangle }e_{k}\right)
+ce_{k}e^{i\left\langle \gamma,x\right\rangle }\right\Vert ^{2}\geq\left(
c_{2}\left\vert \gamma\right\vert ^{2s}+c\right)  ^{2},
\]
where $\{e_{1},e_{2},...,e_{m}\}$ is the standard basis of $\mathbb{R}^{m}$
and $\gamma\in\Gamma$. Using this and the decomposition%
\begin{equation}
u=%
{\textstyle\sum\limits_{\gamma\in\Gamma}}
\left(
{\textstyle\sum\limits_{k=1,2,...,m}}
u_{\gamma,k}e^{i\left\langle \gamma,x\right\rangle }e_{k}\right)  \tag{23}%
\end{equation}
of$\ u$ by the orthonormal basis $\left\{  e^{i\left\langle \gamma
,x\right\rangle }e_{k}:\gamma\in\Gamma,\text{ }k=1,2,...,m\right\}  $, we
obtain that
\begin{equation}
\left\Vert \left(  Lu\right)  +cu\right\Vert ^{2}\geq%
{\textstyle\sum\limits_{\gamma\in\Gamma}}
\left(
{\textstyle\sum\limits_{k=1,2,...,m}}
\left\vert u_{\gamma,k}\left(  c_{2}\left\vert \gamma\right\vert
^{2s}+c\right)  \right\vert ^{2}\right)  . \tag{24}%
\end{equation}
On the other hand, using (23) we have%
\begin{equation}
\left(
{\textstyle\sum\limits_{\left\vert \alpha\right\vert \leq2s-1}}
\left\Vert D_{\alpha}u\right\Vert \right)  ^{2}\leq c_{3}%
{\textstyle\sum\limits_{\left\vert \alpha\right\vert \leq2s-1}}
\left\Vert D_{\alpha}u\right\Vert ^{2}=c_{3}%
{\textstyle\sum\limits_{\gamma\in\Gamma}}
\left(
{\textstyle\sum\limits_{k=1,2,...,m}}
\left\vert u_{\gamma,k}%
{\textstyle\sum\limits_{\left\vert \alpha\right\vert \leq2s-1}}
\left\vert \gamma\right\vert ^{\alpha}\right\vert ^{2}\right)  . \tag{25}%
\end{equation}
Therefore, (22) follows from (24) and (25), if we prove that there exists
$c>0$ such that%
\begin{equation}
c_{2}\left\vert \gamma\right\vert ^{2s}+c>\sqrt{c_{3}}\left(  c_{1}%
+\varepsilon\right)
{\textstyle\sum\limits_{\left\vert \alpha\right\vert \leq2s-1}}
\left\vert \gamma\right\vert ^{\alpha}, \tag{26}%
\end{equation}
for all $\gamma\in\Gamma$. Hence, it remains to prove (26). Clearly, there
exists $c_{4}$ such that $c_{2}\left\vert \gamma\right\vert ^{2s}$ is greater
than the right-hand side of (26), for $\left\vert \gamma\right\vert >c_{4}$.
Besides, the number of $\gamma\in\Gamma$ satisfying $\left\vert \gamma
\right\vert \leq c_{4}$ is finite. Therefore, there exists $c>0$ such that
(26) holds, for $\left\vert \gamma\right\vert \leq c_{4}$. Thus, (26) and
hence (22), (19) and (11) hold.

Now, we prove that (3) holds, too. Let $L_{t}$ and $P_{t}$ be respectively the
operators generated by the differential expression $Lu$ and $Pu$ and boundary
conditions (2). Since the orthonormal basis $\left\{  e^{i\left\langle
\gamma+t,x\right\rangle }e_{k}:\gamma\in\Gamma,k=1,2,...,m\right\}  $ of
$L_{2}^{m}(F)$ is the set of the eigenfunctions of $L_{t}$ and all eigenvalues
of $L_{t}$ are nonnegative numbers, we have $L_{t}\in SBC(L_{2}^{m}(F))$.

Now, we prove that $T_{t}\in SBC(L_{2}^{m}(F))$, that is, (3) holds. It
readily follows from the proof of (26) that, there exists $c>0$ such that%
\[
c_{2}\left\vert \gamma\right\vert ^{2s}+c>2\sqrt{c_{3}}\left(  c_{1}%
+\varepsilon\right)
{\textstyle\sum\limits_{\left\vert \alpha\right\vert \leq2s-1}}
\left\vert \gamma\right\vert ^{\alpha}.
\]
Therefore, arguing as in the proof of (19), we see that, there exist $c_{5}$
such that%
\begin{equation}
\left\Vert Pu+Bu\right\Vert \leq c_{5}\left\Vert u\right\Vert +\tfrac{1}%
{2}\left\Vert Lu+cu\right\Vert . \tag{27}%
\end{equation}
It shows that, $P_{t}+B$ is relatively bounded with respect to $L_{t}+cI$ with
relative bound smaller than $1$. Therefore, by Theorem 4.11 of [2] (Chap.5,
page 291), $T_{t}+cI=L_{t}+cI+P_{t}$ $+B$ is self-adjoint and bounded from
below.\ It remains to prove that $T_{t}+cI$ has a compact resolvent. For this,
we use Theorem 3.17 of [2, Chap. 4, page 214] which implies that, if (27)
holds and there exists $\mu\in\rho(L+cI)$ such that
\begin{equation}
c_{5}\left\Vert \left(  L_{t}+cI-\mu I\right)  ^{-1}\right\Vert +\frac{1}%
{2}\left\Vert \left(  L_{t}+cI\right)  \left(  L_{t}+cI-\mu I\right)
^{-1}\right\Vert <1, \tag{28}%
\end{equation}
then $T_{t}+cI$ has a compact resolvent. Since $L_{t}+cI$ is a self-adjoint
operator, it is clear and well-known that (see, for example (4.9) of [2, Chap.
5, page 290]), if $\mu=ic_{6},$ then%
\[
\left\Vert \left(  L_{t}+cI-\mu I\right)  ^{-1}\right\Vert \leq\frac{1}{c_{6}%
},\text{ }\left\Vert \left(  L_{t}+cI\right)  \left(  L_{t}+cI-\mu I\right)
^{-1}\right\Vert \leq1.
\]
Thus, if $c_{6}>2c_{5}$, then (28) holds. Hence, $T_{t}+cI$ is a below-bounded
self-adjoint operator with a compact resolvent and (3) holds. Now, the proof
follows from Theorem 4.
\end{proof}

Now we show that in some cases, instead of (10), one can constructively define
normalized eigenfunctions $\Psi_{n,t}$ that continuously depend on $t.$ First
of all, let us note the following obvious statement. If $\lambda_{n}(t)$ is a
simple eigenvalue, then the set of all normalized eigenfunctions corresponding
to $\lambda_{n}(t)$ is $\left\{  e^{i\alpha}\Psi_{n,t}:\alpha\in\lbrack
0,2\pi)\right\}  ,$ where $\Psi_{n,t}$ is a fixed normalized eigenfunction. If
the eigenvalue $\lambda_{n}(a)$ is simple, then there exists a neighborhood
$U(a)$ of $a$ such that for $t\in U(a)$ the eigenvalue $\lambda_{n}(t)$ is
also simple and the equality
\begin{equation}
\int_{\gamma}(T_{t}-\lambda I)^{-1}e^{i(a,x)}d\lambda=(e^{i(a,x)},e^{i\alpha
}\Psi_{n,t})e^{i\alpha}\Psi_{n,t}=(e^{i(a,x)},\Psi_{n,t})\Psi_{n,t} \tag{29}%
\end{equation}
is true for any choice of the normalized eigenfunction $\Psi_{n,t},$ where
$\gamma$ is a closed curve that encloses only the eigenvalue $\lambda_{n}(t).$
Since the projection operator onto the subspace corresponding to the
eigenvalue $\lambda_{n}(t)$ continuously depends on $t\in U(a)$ and the\ norm
is a continuous function, it follows from (29) that $\left\vert (\Psi
_{n,t},e^{i(a,x)})\right\vert $ is also continuous function with respect to
$t$ in $U(a)$ for any normalized eigenfunction $\Psi_{n,t}.$ This with the
inequality
\begin{equation}
\left\vert \left\vert (\Psi_{n,t},e^{i(t,x)})\right\vert -\left\vert
(\Psi_{n,t},e^{i(a,x)})\right\vert \right\vert \leq\left\vert (\Psi
_{n,t},e^{i(t,x)})-(\Psi_{n,t},e^{i(a,x)})\right\vert \leq\left\Vert
e^{i(t,x)}-e^{i(a,x)})\right\Vert \tag{30}%
\end{equation}
implies the following obvious statement.

\begin{proposition}
If $\lambda_{n}(a)$ is a simple eigenvalue, then the function $\left\vert
(\Psi_{n,t},e^{i(t,x)})\right\vert $ does not depend on choice of the
normalized eigenfunction $\Psi_{n,t}$ and is continuous in some neighborhood
$U(a)$ of $a$, where $U(a)\subset F^{\ast}$ and any set $E$ satisfying the conditions:

$(a)$ $\left\{  \gamma+t:t\in E,\text{ }\gamma\in\Gamma\right\}
=\mathbb{R}^{d}$ and $(b)$\ if $t\in E,$ then $\gamma+t\notin E$ for any
$\gamma\in\Gamma\backslash\left\{  0\right\}  $ can be used as the fundamental
domain $F^{\star}$ of $\Gamma.$
\end{proposition}

If, in addition, there exist $\varepsilon>0$ such that%
\begin{equation}
\left\vert (\Psi_{n,t},e^{i(t,x)})\right\vert >\varepsilon\tag{31}%
\end{equation}
for all $t\in U(a)$ and the normalized eigenfunction $\Psi_{n,t}$ is chosen so
that
\begin{equation}
\arg(\Psi_{n,t}(x),e^{i(t,x)})=0,\tag{32}%
\end{equation}
then $(\Psi_{n,t},e^{i(t,x)})=\left\vert (\Psi_{n,t},e^{i(t,x)})\right\vert ,$
and hence by Proposition 1, $(\Psi_{n,t},e^{i(t,x)})$ continuously depends on
$t$ in some neighborhood of $a.$ This with (30) and (31) implies that%
\[
\frac{(\Psi_{n,t}(x),e^{i(a,x)})}{(\Psi_{n,a}(x),e^{i(a,x)})}=:\alpha
(t)\rightarrow1
\]
as $t\rightarrow a.$ Therefore, using (31) and the continuity of the right
side of (29) we obtain
\[
\left\Vert (e^{i(a,x)},\Psi_{n,t})\Psi_{n,t}-(e^{i(a,x)},\Psi_{n,a})\Psi
_{n,a}\right\Vert \rightarrow0
\]
and $\left\Vert \alpha(t))\Psi_{n,t}-\Psi_{n,a}\right\Vert \rightarrow0$ as
$t\rightarrow a.$ Thus we have
\[
\left\Vert \Psi_{n,t}-\Psi_{n,a}\right\Vert \leq\left\Vert (1-\alpha
(t))\Psi_{n,t}\right\Vert +\left\Vert \alpha(t))\Psi_{n,t}-\Psi_{n,a}%
\right\Vert \rightarrow0
\]
as $t\rightarrow a.$ In other words, the following statement is proved.

\begin{proposition}
If $\lambda_{n}(a)$ is a simple eigenvalue and (31) holds, then the normalized
eigenfunction $\Psi_{n,t}$ satisfying (32) continuously depends on $t$ in some
neighborhood of $a.$
\end{proposition}

\begin{remark}
Note that the constructive choice (32) is also used in [6]. Namely in [6] for
the case when the right side of (1) is equal to $-\Delta+q$ and $m=1$ (for the
Schr\"{o}dinger operator $L$) in the neighborhood of the sphere $\left\{
t\in\mathbb{R}^{d}:\left\vert t\right\vert =\rho\right\}  ,$ where $\rho$ is a
large number, I constructed a set $B$ such that if $t\in B,$ then there exists
unique eigenvalue $\lambda_{n(t)}(t)$ that is simple and close to $\left\vert
t\right\vert ^{2}$ and the corresponding normalized eigenfunction $\Psi_{n,t}$
satisfies the asymptotic formula
\begin{equation}
\left\vert (\Psi_{n,t},e^{i(t,x)})\right\vert ^{2}=1+O(\rho^{-\delta}%
)>\tfrac{1}{2} \tag{33}%
\end{equation}
for some $\delta>0.$ Moreover, the normalized eigenfunction was chosen so that
(32) holds (see [6] page 55). In [6] the choice (32) was made in order to
write (33) in an elegant form $\Psi_{n,t}=e^{i(t,x)}+O(\rho^{-\delta})$ .
However, in Proposition 2 we show that the choice (32) ensures the continuity
of $\Psi_{n,t}.$ Note that Proposition 2 is also new for the Schr\"{o}dinger
operator. However, Proposition 1 for the Schr\"{o}dinger operator is clear,
since it follows directly from the continuities of the function $e^{i(t,x)},$
the projection operator, and the norm. Proposition 1 and (33) were used to
prove that $n(t)=n(a)$ for all $t\in U(a)$ (see (5.11) from [6]), where
$U(a)\subset B$ and satisfies condition $(b)$ (see Lemma 5.1$(b)$ from [6] and
Proposition 1), i.e., $U(a)\subset\left(  B\cap F^{\star}\right)  $ for some
fundamental domain $F^{\ast}.$
\end{remark}

\begin{acknowledgement}
The work was supported by the Scientific and Technological Research Council of
Turkey (Tubitak, project No. 119F054).
\end{acknowledgement}

\end{document}